\documentclass[reqno]{amsart}
\usepackage{amsmath,amssymb,cite}
\usepackage[mathscr]{euscript}
\usepackage{mathtools}
\usepackage{xcolor}
\usepackage{stix}
\usepackage{hyperref}
\usepackage[final]{showkeys}
\usepackage{dsfont}
\usepackage{relsize}
\usepackage[normalem]{ulem}
\usepackage{bbold,dsfont}
\usepackage{tikz}

\usepackage[final]{fixme}

\definecolor{bblue}{rgb}{.2,0.2,.8}

\theoremstyle{plain}
\newtheorem{theorem}{Theorem}[section]
\newtheorem{proposition}[theorem]{Proposition}
\newtheorem{lemma}[theorem]{Lemma}
\newtheorem{corollary}[theorem]{Corollary}

\theoremstyle{definition}
\newtheorem{definition}[theorem]{Definition}

\theoremstyle{remark}
\newtheorem{remark}[theorem]{Remark}

\numberwithin{equation}{section}
\numberwithin{theorem}{section}

\newcommand*{\defeq}{\mathrel{\vcenter{\baselineskip0.5ex \lineskiplimit0pt
			\hbox{\scriptsize.}\hbox{\scriptsize.}}}%
	=}

    \let\phi=\varphi

\renewcommand{\epsilon}{\varepsilon}
\renewcommand{\phi}{\varphi}
\renewcommand{\hat}{\widehat}

\newcommand{\<}{\langle}
\newcommand{\Set}[1]{\left\{#1\right\}}
\renewcommand{\>}{\rangle}

\DeclareMathOperator{\Fun}{Fun}

\DeclareMathOperator{\AGL}{AGL}

\DeclareMathOperator{\Sym}{Sym}
\DeclareMathOperator{\F}{\mathbb{F}}
\DeclareMathOperator{\Z}{\mathbb{Z}}

\DeclareMathOperator{\pdeg}{pdeg}
\DeclareMathOperator{\Lie}{\mathfrak{L}}
\DeclareMathOperator{\lt}{\mathrm{lt}}

\title{Normality conditions in the Sylow \(\boldsymbol{p}\)-subgroup of \(\boldsymbol{\Sym(p^n)}\) and its associated Lie algebra}

\author{Riccardo Aragona}
\address{\noindent Riccardo Aragona \hfill\break\indent 
	DISIM, Universit\`a dell'Aquila
	\hfill\break\indent 
	67100 Coppito, L'Aquila, Italy
}
\email{riccardo.aragona@univaq.it}

\author{Norberto Gavioli}
\address{\noindent Norberto Gavioli \hfill\break\indent 
	DISIM, Universit\`a dell'Aquila
	\hfill\break\indent 
	67100 Coppito, L'Aquila, Italy
}
\email{norberto.gavioli@univaq.it}

\author{Giuseppe Nozzi}
\address{\noindent Giuseppe Nozzi \hfill\break\indent 
	DISIM, Universit\`a dell'Aquila
	\hfill\break\indent 
	67100 Coppito, L'Aquila, Italy
}
\email{giuseppe.nozzi@graduate.univaq.it}

\begin{document}
	
\thanks{All the authors are members of INdAM-GNSAGA (Italy); the first
  and  the second  authors are  members  of the  center of  excellence
  ExEmerge of the University of L'Aquila. The authors thankfully acknowledge support by MUR-Italy via PRIN 2022RFAZCJ "Algebraic methods in Cryptanalysis".}
\begin{abstract}
 In this work, we give a description of the structure of the normal subgroups of a Sylow \(p\)-subgroup \(W_n\) of \(\Sym(p^n)\), showing that they contain a term from the lower central series with bounded index. To this end,  we explicitly determine the terms of the upper and the lower central series of \(W_n\). We provide a similar description of these series in the  Lie algebra associated to \(W_n\), giving a new proof of the equality of their terms in both the group and algebra contexts. Finally, we calculate the growth of the normalizer chain starting from an elementary abelian regular subgroup of \(W_n\). 
 \end{abstract}
	
\noindent

\maketitle
\thispagestyle{empty}
\section{Introduction}
This work studies normality conditions in a Sylow \(p\)-subgroup \(W_n\) of the symmetric group \(\Sym(p^n)\) and relations with the ideals of the Lie algebra associated to the lower central series of this group. 

This article aims also to be the conclusion of a series of papers focused on the study  of a sequence of normalizers arising from an abelian regular subgroups of the Sylow \(p\)-subgroup of \(\Sym(p^n)\). More specifically, let \(T\) be a vector space of dimension \(n\) over \(\mathbb{F}_p\). This group acts regularly on itself, so it can be seen as an elementary abelian regular subgroup of \(\Sym(p^n)\) via the Cayley embedding. In~\cite{regular}, the authors prove that any other elementary abelian regular subgroup of \(\AGL(2^n)\) intersecting $T$ in a second-maximal subgroup is conjugated to \(T\) in \(N_{W_{n}}(T)\). The latter normalizer  is also equal to \(N_{\Sym(2^n)}(T)\) in the sole case \(p=2\). In~\cite{chain,rigid,unref}, a chain of normalizers originating from the subgroup \(T\) is defined as follows
	\begin{equation*}
	N_i=\begin{cases}
		T&\text{ for }i=-1\\
		N_{W_n}(N_{i-1})&\text{ for }i\ge 0
	\end{cases}
\end{equation*}
In the case \(p=2\), the authors give the notion of rigid commutators in order to study the growth of this chain.
This approach enables them to show that the relative indices of consecutive terms of this chain are related to the number of unrefinable partitions into at least two distinct parts. More precisely, they prove that \(\log_2(|N_i:N_{i-1}|)=q_{2,i+2}\), where  \(\{q_{2,i}\}_{i\ge 1}\) is the sequence of partial sums of the sequence  of the partitions into at least two distinct parts of \(i\).
A \((\Z/m\Z)\)-Lie ring counterpart of these results is provided in~\cite{modular}, where the  Lie ring \(\Lie_m(n)\) of partitions with maximal part at most \(n-1\) is constructed together with a chain of idealizer \(\{\mathfrak{N}_i\}_{i\ge -1}\) mimicking the aforementioned normalizer chain. When \(m=2\) the idealizer chain growth is shown to be the same as in the group context.
When \(m=p\) is an odd prime the rank of \(\mathfrak{N}_i/\mathfrak{N}_{i-1}\) is shown to be equal to \(q_{m,i+1}\), i.e., the \((i+1)\)-th term of the sequence of partial sums of the sequence  of the partitions into at least two distinct parts of \(i\), where each part can be repeated at most \(m-1\) times. The authors of the cited paper conjecture that a similar result holds for the normalizer chain in the Sylow \(p\)-subgroup of \(\Sym(p^n)\), without being able to prove it, as they face difficulties in defining a suitable notion of rigid commutators for odd primes. In the present work,  we find an alternative way, with respect to the rigid commutators used in characteristic 2, of describing the elements of the Sylow \(p\)-subgroup and we successfully prove the above conjecture.  

Let \(p>2\) be a prime and \(W_n\) be the Sylow \(p\)-subgroup of \(\Sym(p^n)\). This group can be seen as the iterated wreath product \(W_n=\wr_{i=1}^n \F_p\) (see~\cite{kaloujnine2}). The techniques involved in the present paper make a large use of the central series of the group \(W_n\). This topic is  widely studied  in literature (see e.g.~\cite{meldrum,kaloujnine,polnum,neumann,commutatorcalculus,sushchanskii}). In particular, Kaloujnine in~\cite{kaloujnine} proves that the upper and the lower central series of a Sylow \(p\)-subgroup of \(\Sym(p^n)\) coincide. Sushchanskyy in~\cite{sushchanskii}  shows that the same result holds also for wreath products of elementary abelian groups. 

\smallskip

We now give a brief outline of the paper. After describing in Section~\ref{sec:prel} some useful preliminaries, in Section~\ref{sec:lowup} we compute the upper and the lower central series of \(W_n\) and we provide a proof, alternative to the Kaloujnine's one, that these two series coincide. 


In Section~\ref{sec:normal}, we look at the normal subgroups \(N\) of \(W_n\). Specifically, we prove that if \(N\) is contained in the last \(n-k\) base subgroups of \(W_n\), then it contains a term of the lower central series with bounded index depending only on \(k\).

In Section~\ref{sec:lie}, we introduce the graded Lie algebra \(\Lie_n\) associated to the lower central series of \(W_n\). In~\cite{netreba}, the authors characterize this Lie algebra as the iterated wreath product \(\Lie_n=\wr^n \Lie_1\), where \(\Lie_1\) is the one-dimensional algebra over \(\mathbb{F}_p\). We introduce a map \(\phi\colon W_n\to \Lie_n\)  establishing a relation between the group and the algebra and which intertwines  central series and a special class of normal subgroups with homogeneous ideals. In particular, also the upper and lower central series of \(\Lie_n\) coincide. Finally, in Section~\ref{sec:chain}, using the map \(\phi\), we relate the normalizer chain originating from the canonical regular elementary abelian subgroup of \(W_n\) and the idealizer chain introduced in~\cite{modular}, proving that they exhibit the same growth.

\section{Preliminaries}\label{sec:prel}
In this section, we introduce some notations and definitions that will be used throughout the paper.
\subsection{Wreath product}
Let \(W_n=\wr^n \F_p\), where \(p\) is an odd prime integer. Notice that \(W_n=\Fun(\F_p^{n-1}, \F_p) \rtimes W_{n-1}\) and that the group of functions \(\Fun(\F_p^{n-1}, \F_p)\) can be identified with the additive group of the polynomials in \(n-1\) variables in which every variable appears with degree at most \(p-1\).  More precisely, the \(k\)-th base subgroup of \(W_n\) is defined as
  \[B_k\defeq\Fun(\F_p^{k-1}, \F_p) \cong \F_{p}[x_1,\dots, x_{k}]/(x_1^p-x_1, \dots, x_{k-1}^p-x_{k-1})\]
and \(W_n= \rtimes_{k=0}^{n-1}B_{n-k}, \). We will also denote a polynomial \(f\in B_k\) as \(f\Delta_k\) to avoid confusion, since the same polynomial may belong to different base subgroups. In particular, every element of \(w\in W_n\) can be uniquely written as a product of the form \(w= \prod_{k=0}^{n-1} f_{n-k}\Delta_{n-k}\), where \(f_i\in \Fun(\F_p^{i-1}, \F_p)\).

 Let $i$ and $k$ be integers such that \(i<k\). Let \(x=(x_1,\dots,x_{k-1})\)  and \(e_i\)  be the \(i\)-th  vector of  the canonical
basis  of  \(\F_p^{k-1}\).   For  each  $h\Delta_i\in  B_i$,  we  define  
$$ \Delta_i(h)(f  (x)\Delta_k)= (f(x+he_i)-f(x))\Delta_k.$$  
The operator $\Delta$  can be
used to express   the conjugation action of an element \(f_i\Delta_i\in B_i\) on an element \(f_k\Delta_k \in B_k\)  by way of the commutator \[
[f_k\Delta_k,{f_i\Delta_i}]= \Delta_i(f_i)\bigl( f_k\Delta_k \bigr) .
\] 

Since the functions in the base subgroups are polynomials in which  every variable appears with degree at most \(p-1\), we can use Taylor formula to write the above commutator as follows
\begin{equation}\label{taylor}
[f_k\Delta_k,{f_i\Delta_i}]= \sum_{j=1}^{p-1}\dfrac{1}{j!}\dfrac{\partial^j f_k}{ \partial x_i^j}{f_i^j}\Delta_k.
\end{equation}

 Kaloujnine in \cite{kaloujnine} proved that \(W_n\) is the Sylow \(p\)-subgroup of \(\Sym(p^n)\). The element \(f_i\Delta_i\in B_i\) acts on \((x_1,\dots, x_n)\in \F_p^{n}\) via the translation  
%
\[
	(x_1,\dots, x_n) \to (x_1,\dots, x_n) - e_i  f_i(x_1,\dots,x_{i-1})
\] 
In this action, the subgroup \(T=\< \Delta_1,\dots, \Delta_n \> \le W_n\) acts regularly and is called \emph{the canonical elementary abelian regular subgroup of \(W_n\)}. 

\subsection{Power monomials}
Let $\Lambda=\{\lambda_i\}_{i=1}^\infty$  be  a sequence  of  non-negative
integers with finite support and weight
\begin{equation}\label{defwt}
	\mathrm{wt}(\Lambda)\defeq \sum_{i=1}^\infty i\lambda_i< \infty.
\end{equation}
We   shall   say  that   $\Lambda$   is   a   partition  of   $N$   if
$\mathrm{wt}(\Lambda)=N$. The maximal part of $\Lambda$ is the integer
$\max\{i\mid  \lambda_i\neq 0\}$.   The  set of  the partitions  whose
maximal  part   is  less  than   or  equal   to  $k$  is   denoted  by
$\mathcal{P}(k)$ and we define for each $m\ge 1$
\begin{equation*}
	\mathcal{P}_m(k)=\Set{\Lambda\in \mathcal{P}(k)\mid \lambda_i\le m-1\text{ for\ all } i}.
\end{equation*} 
Let $\Lambda\in \mathcal{P}(k)$. We define the \emph{power monomial}  $x^\Lambda$ by
\begin{equation*}
	x^\Lambda=\prod_{i=1}^\infty x_i^{\lambda_i}.
\end{equation*}
Throughout this paper, unless otherwise  stated, we will consider partitions in $ \mathcal{P}_p(k)$ for $k=1,\dots,n$. The set $\mathcal{B}$ consists of all power monomials in \(W_n\), i.e.
\begin{equation}
	\mathcal{B}=\Set{x^\Lambda\Delta_k\mid \Lambda\in \mathcal{P}_p(k)\text{ and }1\le k\le n}.
\end{equation}
\begin{definition}\label{def:pdeg}
	 We define the  \emph{$p$-degree}  of the power monomial $x^\Lambda=x_1^{\lambda_1}\cdots x_{n-1}^{\lambda_{n-1}}$,
	written $\pdeg(x^\Lambda)$, by
	\begin{equation*}
		\pdeg(x^\Lambda)=\lambda_{n-1}p^{n-2}+\dots+\lambda_2p+\lambda_1.
	\end{equation*}
Moreover, let  $1\le j\le n$, we set \(\mu_j=p^{n-1}-p^{j-1}\) and we define
\begin{equation}
	\pdeg(x^\Lambda\Delta_j)=\pdeg(x^\Lambda)+\mu_j.
\end{equation}
 Let $f\Delta_k=(c_1x^{\Lambda_1}+\dots+c_tx^{\Lambda_t})\Delta_k\in B_k$ be an homogeneous element with $c_i\in \F_p$. We define \(\pdeg(f\Delta_k)\) as the \(\max\Set{\pdeg(x^\Lambda_i\Delta_k)\mid i=1,\dots,t }\) and refer to the \emph{leading term} $\lt(f\Delta_k)$ of $f\Delta_k$ as the element of \(W_n\) which realizes that maximum.
\end{definition}
Notice that if $\pdeg(x^\Lambda\Delta_k)<\mu_k$, then $x^\Lambda\Delta_k=1$.
\medskip

The following definition is already given in~\cite{wreath0}.
\begin{definition}\label{def:saturated}
	A subgroup \(S\le W_n\)  is said to be saturated if \begin{enumerate}
		\item \(S=S_1\cdots S_n\), where \(S_i\le B_i\),
		\item if \(f\Delta_k\in S\), then for each monomial \(cx^\Lambda\) of \(f\), with \(c\in \F_p^*\), the element \(x^\Lambda\Delta_k\) is in \(S\).
	\end{enumerate}
\end{definition}
Notice that a saturated subgroup $S$ of $W_n$ is spanned by the set $S\cap \mathcal{B}$.

\section{The Lower and the upper Central Series of $W_n$}\label{sec:lowup}
In this section we compute the lower and the upper central series of $W_n$ and we give a proof of the equality between the terms of the two series.

 Let us consider $x^\Lambda\Delta_k\in B_k$ and $x^\Theta\Delta_\ell\in B_\ell$ with $k>\ell$. By \cite[Lemma~2.8]{wreath0}, it is easy to see   that

	\begin{equation*}
		\pdeg([x^\Lambda\Delta_k,x^\Theta\Delta_\ell])=\pdeg\left(\dfrac{\partial x^\Lambda}{\partial x_\ell}x^\Theta \Delta_k \right).
	\end{equation*}
	Thus, if $[x^\Lambda\Delta_k,x^\Theta\Delta_\ell]\neq 0$, then 
	\begin{equation}\label{rmk:commutators}
	\pdeg([x^\Lambda\Delta_k,x^\Theta\Delta_\ell])< \pdeg(x^\Lambda\Delta_k).
	\end{equation}
\begin{lemma}\label{lem:pdeg-1}
	Let $x^\Lambda\Delta_k\in B_k$. There exists a monic monomial element $w\in W_n$ such that $[x^\Lambda\Delta_k,w]$ lies in $B_k$ and
	$$\pdeg([x^\Lambda\Delta_k,w])=\pdeg(x^\Lambda\Delta_k)-1.$$
\end{lemma}
\begin{proof}
	Let $\Lambda=(\lambda_1,\dots,\lambda_{k-1})$ and $j=\min\lbrace j\mid \lambda_j\neq 0\rbrace.$ If $x_j=x_1$, then $w=\Delta_1$. Indeed, we have that
	$$\pdeg([x^\Lambda\Delta_k,\Delta_1])=\pdeg\bigl(\dfrac{\partial x^\Lambda}{\partial x_1}\Delta_k\bigr)=\pdeg(x^\Lambda\Delta_k)-1.$$
If $x_j\neq x_1$, then $w=x_1^{p-1}\cdots x_{j-1}^{p-1}\Delta_j$. Indeed
\begin{align*}
\pdeg([x^\Lambda\Delta_k,x_1^{p-1}\cdots x_{j-1}^{p-1}\Delta_j])&=\pdeg\bigl(\dfrac{\partial x^\Lambda}{\partial x_j}x_1^{p-1}\cdots x_{j-1}^{p-1}\Delta_k\bigr)\\
&=(\pdeg(x^\Lambda)-p^{j-1})+(p-1)\sum_{i=0}^{j-2}p^i+\mu_k\\
&=\pdeg(x^\Lambda\Delta_k)-1.\qedhere
\end{align*}
\end{proof}
\begin{corollary}\label{cor:commutators}
	If $x^\Lambda\Delta_k\in B_k$, then there exist monic monomial elements $w_1,\dots,w_\ell\in W_n$ such that the commutator $[x^\Lambda\Delta_k,w_1,\dots,w_\ell]$ lies in $B_k$ and
	\begin{equation}
		\pdeg([x^\Lambda\Delta_k,w_1,\dots,w_\ell])=\pdeg(x^\Lambda\Delta_k)-\ell
	\end{equation}
	where \(1\le \ell\le \pdeg(x^\Lambda)\).
\end{corollary}
We use the standard notation $\gamma_i(W_n)$ to indicate the $i$-th term of the lower central series of $W_n$.
\begin{lemma}\label{lem:gammasuiB}
	Let $i\ge 1$, then $\gamma_i(W_n)\cap B_k=\langle x^\Lambda\Delta_k\mid \pdeg(x^\Lambda\Delta_k)\le p^{n-1}-i\rangle$.
\end{lemma}
\begin{proof}
	By Equation~\eqref{rmk:commutators} and Corollary~\ref{cor:commutators}, it is enough to notice that $\max\lbrace \pdeg(x^\Lambda\Delta_k)\mid x^\Lambda\Delta_k\in \gamma_i(W_n)\cap B_k\rbrace =p^{n-1}-i.$ That maximum is reached by applying Corollary~\ref{cor:commutators} to the maximum element $x_1^{p-1}\cdots x_{k-1}^{p-1}$ of $B_k$. Moreover, by applying Corollary~\ref{cor:commutators} to the monic monomial element of $B_k$ with $\pdeg$ equal to $s+i$, we get  the monic monomial element with $\pdeg$ equal to $s$ in $\gamma_i(W_n)\cap B_k$,  for each $s<p^{n-1}-i$. 
\end{proof}

\begin{lemma}
	Let \(A\) an abelian subgroup of \(G\) and \(B\trianglelefteq G\) such that
	\begin{enumerate}
		\item \(AB=A \ltimes B\) is normal in \(G\),
		\item there exists \(H\le G\) such that \(H\le N_G(A)\) and \(G=H(AB)\).
	\end{enumerate}
	Then \([AB,G] = ([A,G]\cap A) [B,G]\).
\end{lemma}
\begin{proof}
	Let \(g=h\bar a\bar b\in G\) with \(h\in H\), \(\bar a \in A\), \(\bar b \in B\) and \(ab\in AB\). Note that the commutator
\begin{eqnarray*}
	[ab,g]=[a,g][a,g,b][b,g]\in [a,g][B,G].
\end{eqnarray*}
Moreover, since  \(A\) is abelian we have
\begin{equation*}
	[g,a]=[h,a][h,a,\bar b][\bar b,a] \in ([G,A]\cap A)[B,G]
\end{equation*}
Hence \([ab,g]   \in ([G,A]\cap A)[B,G]\), so \([AB,G]\le ([G,A]\cap A)[B,G]\). Since the opposite inclusion is trivial we have the claim.
\end{proof}
\begin{corollary}
	In the same hypotheses oh the previous lemma the following equality holds 
	\[[AB,\underbrace{G,\dots,G}_{k\text{ times}}] = ([A,\underbrace{G,\dots,G}_{k\text{ times}}]\cap A) [B,\underbrace{G,\dots,G}_{k\text{ times}}]\]
\end{corollary}
The following is another straight consequence
\begin{corollary}\label{cor:desc_satur}
	$\gamma_i(W_n)= (\gamma_i(W_n)\cap B_n)\rtimes\dots \rtimes (\gamma_i(W_n)\cap B_1) $.
\end{corollary}
Notice that the terms of the lower central series of \(W_n\) are saturated subgroups and  $$\gamma_i(W_n)=\gamma_{i+1}(W_n)\rtimes\langle c x^\Lambda\Delta_k\mid \pdeg(x^\Lambda)=p^{n-1}-i,\ 1\le k\le n\text{ and }c\in \F_p\rangle.$$

The remaining part of this section is devoted to prove the following theorem.
\begin{theorem}\label{thm:upperlower}
	The upper and the lower central series of $W_n$ coincide.
\end{theorem}

To prove this theorem, we need to prove some preliminary results.

\begin{lemma}\label{lem:upperlowerBn}
	\(Z_i(W_n)\cap B_n = \gamma_{p^{n-1}-i}(W_n)\cap B_n = [(Z_{i+1}(W_n)\cap B_n), W_n]\).
\end{lemma}
\begin{proof}
	We know that \(Z_i(W_n)\cap B_n \ge  \gamma_{p^{n-1}-i}(W_n)\cap B_n\). 
	Since \(B_n\) is a uniserial \(W_n\)-module, for all \(i,j\ge 1\) we have the followings 
	\begin{enumerate}
		\item \(\gamma_{j}(W_n)\cap B_n = [B_n,\underbrace{W_n,\dots,W_n}_{\text{\(j-1\) times}}]\);
		\item \(|(\gamma_{p^{n-1}-i}(W_n)\cap B_n) :  (\gamma_{p^{n-1}-i-1}(W_n)\cap B_n|=p)\);
		\item \(|Z_i(W_n): Z_{i-1}(W_n)| \ge p\);
		\item  \(Z_{p^{n-1}-1}(W_n)\cap B_n\ne B_n\); 
		\item \(Z_{p^{n-1}}(W_n)\cap B_n=B_n\).
	\end{enumerate}
	Thus, \(p^{n-1}=\prod_{j=1}^{p^{n-1}}|(Z_{j}(W_n)\cap B_n): (Z_{j-1}(W_n) \cap B_n)| \ge \prod_{j=1} p =p^{n-1}\) and so \(|(Z_{j}(W_n) \cap B_n): (Z_{j-1}(W_n) \cap B_n)|=p\) for all \(j\). The statement follows inductively noting that
	\(Z_{1}(W_n)\cap B_n = \gamma_{p^{n-1}}(W_n)\cap B_n\) 
\end{proof}

\begin{lemma}
	If \(g=g_k\dots g_n\in Z_i(W_n)\) with \(g_s\in B_s\) then \(g_k\in Z_i(W_n)\).
\end{lemma}
\begin{proof}
	Let \(\psi \colon W_{n} \to W_{k}\) be the canonical map whose kernel is \(B_{k+1}\cdots B_n\). Note that \(\psi(Z_i(W_n)) \le Z_{i-p^{n-1}+p^{k-1}}(W_k)\). By Lemma~\ref{lem:upperlowerBn} it follows that \[\psi(g)=g_k\in Z_{i-p^{n-1}+p^{k-1}}(W_k)\cap B_k= \gamma_{p^{n-1}-i}(W_k)\cap B_k= \gamma_{p^{n-1}-i}(W_n)\cap B_k.\]
	Thus, \(g_k\in \gamma_{p^{n-1}-i}(W_n)\cap B_k \le Z_{i}(W_n)\cap B_k\).
\end{proof}
\begin{corollary}\label{cor:upperlower}
	If \(g=g_1\dots g_n\in Z_i(W_n)\) with \(g_s\in B_s\) then \(g_s\in Z_i(W_n)\) for all \(s\). In particular
	\[Z_{i}(W_n)= (Z_{i}(W_n)\cap B_1) \cdots  (Z_{i}(W_n)\cap B_n).\]
\end{corollary}
\begin{proof}[Proof of Theorem~\ref{thm:upperlower}]
	Notice that \[(Z_{i}(W_n)\cap B_1) \cdots  (Z_{i}(W_n)\cap B_n)= (\gamma_{p^{n-1}-i}(W_n)\cap B_1) \cdots  (\gamma_{p^{n-1}-i}(W_n)\cap B_n).\]
	Indeed, by Lemma~\ref{lem:upperlowerBn}, we have that
	\(
	(Z_{i}(W_n)\cap B_n)=\gamma_{p^{n-1}-i}(W_n)\cap B_n
	\)
	and the claim follows easily by arguing induction considering the quotient \(W_{n-1}=W_n/B_n\). 
	
	Finally, by Corollaries~\ref{cor:desc_satur} and~\ref{cor:upperlower}, we have
	\begin{multline*}
		Z_{i}(W_n)= (Z_{i}(W_n)\cap B_1) \cdots  (Z_{i}(W_n)\cap B_n)\\
		=(\gamma_{p^{n-1}-i}(W_n)\cap B_1) \cdots  (\gamma_{p^{n-1}-i}(W_n)\cap B_n)=
		\gamma_{p^{n-1}-i}(W_n).
		\qedhere
	\end{multline*}
\end{proof}

\section{Normal Subgroups of \(W_n\)}\label{sec:normal}
In this section we characterize the normal subgroups of \(W_n\) showing that they contain a term of the lower central series with bounded index.
\begin{lemma}\label{lem: max in N}
Let $N$ be a normal subgroup of \(W_n\) and $f\Delta_k\in N$. Every monomial element \(x^{\Lambda}\Delta_k\) of \(p\)-degree at most \(\pdeg(f\Delta_k)\) belongs to \(N\). 
\end{lemma}
\begin{proof}
We argue by induction on the \(p\)-degree of \(f\Delta_k\). The base of the induction is when \(f\) is constant of \(p\)-degree minimum possible  \(\mu_k\). In this case, the claim is trivial. Otherwise \(t=\pdeg(f\Delta_k)\geq \mu_k\). By Lemma~\ref{lem:pdeg-1} applied to the leading term \(x^{\Lambda}\Delta_k\) of \(f\Delta_k\), there exists an element in \(N\) of \(p\)-degree equals to \(t-1\). By induction, \(N\cap B_k\) contains every monomial element of \(p\)-degree at most \(t-1\). In particular, \(f\Delta_k-x^{\Lambda}\Delta_k\) lies in \(N\) and so also \(x^{\Lambda}\Delta_k\in N\), proving the statement.
\end{proof}

\begin{lemma}\label{lem:splitnormal}
If \(N\trianglelefteq W_n\), then \((N\cap B_k)(N\cap B_{k+1})\cdots (N\cap B_n) \trianglelefteq W_n\) for all \(1\le k\le n\).
\end{lemma}
\begin{proof}
	Since the elements in \(N\) of the form \(f\Delta_h\), where \(h\ge k\), generate  \((N\cap B_k)(N\cap B_{k+1})\cdots (N\cap B_n)\), it suffices to note  that the commutator \([f\Delta_h, x^\Lambda\Delta_s]\) belongs to \( N\cap B_\ell\) for each generator  \(x^\Lambda\Delta_s\) of \( W_n\), where \(\ell= \max(s,h)\). 
\end{proof}

\begin{lemma}
	Let \(N\) be the normal closure of \(\<f\Delta_k\>\). Then 
	\begin{equation*}
		N= (N\cap B_k)(N\cap B_{k+1})\cdots (N\cap B_n).
	\end{equation*} 
\end{lemma}

\begin{proof}
	On the one hand, note that 
	\(	N\ge (N\cap B_k)(N\cap B_{k+1})\cdots (N\cap B_n)
	\). On the other hand, by Lemma~\ref{lem:splitnormal}, \((N\cap B_k)(N\cap B_{k+1})\cdots (N\cap B_n)\) is a normal subgroup of \(W_n\) containing \(f\Delta_k\), hence it contains its normal closure \(N\). Thus we have the equality.
\end{proof}

\begin{lemma}\label{lem:cyclic_commutator_group}
	If \(1\ne x^\Lambda\Delta_k\in \gamma_{t}(W_n)\setminus \gamma_{t+1}(W_n)\), then \[
	[\<x^\Lambda\Delta_k \>, W_n] =  \bigl(\gamma_{t+1}(W_n) \cap B_{k}\bigr)\bigl(\gamma_{p^{k-1}+1}(W_n) \cap (B_{k+1}\cdots B_n)\bigr).
	\]
\end{lemma}
\begin{proof}
	The inclusion  \([\<x^\Lambda\Delta_k \>, W_n] \le\bigl(\gamma_{t+1}(W_n) \cap B_{k}\bigr)\bigl(\gamma_{p^{k-1}+1}(W_n) \cap (B_{k-1}\cdots B_n)\bigr) \) is trivial. In order to prove the opposite inclusion,  consider  the monomial element \(x^\Theta\Delta_h\in \bigl(\gamma_{t+1}(W_n) \cap B_{k}\bigr)\bigl(\gamma_{p^{k-1}+1}(W_n) \cap (B_{k-1}\cdots B_n)\bigr)  \). Let us first analyze the case \(h=k\). We know that \(\pdeg(x^\Lambda\Delta_k)= \pdeg(x^\Lambda)+\mu_k= p^{n-1}-t\) and \(\pdeg(x^\Theta\Delta_k)=\pdeg(x^\Theta)+\mu_k\le p^{n-1}-t-1 =\pdeg(x^\Lambda\Delta_k)-1\). By Lemma~\ref{lem: max in N}, \(\gamma_{t+1}(W_n) \cap B_{k} \leq [\<x^\Lambda\Delta_k \>, W_n]\).

	 If \(h>k\) it suffices to consider the commutator \[[x^\Lambda\Delta_k,x_1^{p-1-\lambda_1}\cdots x_{h-1}^{p-1-\lambda_{h-1}}\Delta_h]\] which has \(\pdeg\) equal to \(\mu_k-1\) and apply Lemma~\ref{lem: max in N} as above.
\end{proof}

\begin{proposition}\label{prop:normalclosuremonomial}
	The normal closure of the subgroup of \(W_n\) generated by \(x^\Lambda\Delta_k\) is 
	\[\<x^\Lambda\Delta_k\>[\<x^\Lambda\Delta_k\>, W_n]= \<x^\Lambda\Delta_k\>\bigl(\gamma_{t+1}(W_n) \cap B_{k}\bigr)\bigl(\gamma_{p^{k-1}+1}(W_n) \cap (B_{k+1}\cdots B_n)\bigr),\] 
	where \(t=\pdeg(x^\Lambda\Delta_k)\).
\end{proposition}
\begin{proof}
	Since by Lemmas~\ref{lem:splitnormal} and \ref{lem:cyclic_commutator_group} the subgroup \([\<x^\Lambda\Delta_k\>, W_n]\) is normal in \(W_n\), the claim follows.
	%
	%
\end{proof}
The proposition above together with Lemma~\ref{lem: max in N} give the following results.
\begin{corollary}\label{cor:normal closure fdeltak}
	The normal closure of the subgroup of \(W_n\) generated by \(f\Delta_k\) is the saturated subgroup 
	\[\<f\Delta_k\>[\<f\Delta_k\>, W_n]= \<f\Delta_k\>\bigl(\gamma_{t+1}(W_n) \cap B_{k}\bigr)\bigl(\gamma_{p^{k-1}+1}(W_n) \cap (B_{k+1}\cdots B_n)\bigr),\] 
	where \(t=\pdeg(f\Delta_k)\). 
\end{corollary}
\begin{proposition}
	Let \(g=f\Delta_k\cdot h\), with \(h\in B_{k+1}\cdots B_n\). The normal closure \(\langle g\rangle^{W_n}\) contains \(\gamma_{p^{k-1}+1}(W_n)\).
\end{proposition}
\begin{proof}
	If \(h=1\) the statement follows by Corollary~\ref{cor:normal closure fdeltak}.If \(h\neq 1\), let \(x^\Lambda\Delta_k\) be the leading term of \(f\Delta_k\). Observe that \([x_1^{p-1-\lambda_1}\dots x_{n-1}^{p-1-\lambda_{n-1}}\Delta_n,g]\) has \(p\)-degree \(p^{n-1}-p^{k-1}-1\) and so we can apply Lemma~\ref{lem: max in N} to get \(\gamma_{p^{k-1}+1}(W_n)\cap B_n\le \langle g\rangle^{W_n}\). 	Next, the commutator \([x_1^{p-1-\lambda_1}\cdots x_{n-2}^{p-1-\lambda_{n-2}}\Delta_{n-1},g]=s\Delta_n q\Delta_{n-1}\) is such that \[\pdeg(s\Delta_n)\le\pdeg([x_1^{p-1-\lambda_1}\cdots x_{n-2}^{p-1-\lambda_{n-2}}\Delta_{n-1},x_1^{p-1}\dots x_{n-1}^{p-1}\Delta_n])\le p^{n-1}-p^{k-1}-1.\]
	It follows that \(s\Delta_n \in \langle g\rangle^{W_n}\) by the argument above. In particular, \(q\Delta_{n-1}\in \langle g\rangle^{W_n}\) and has \(p\)-degree equal to \(p^{n-2}-p^{k-1}-1\). Thus, by Lemma~\ref{lem: max in N}, \(\gamma_{p^{k-1}+1}(W_n)\cap B_{n-1}\le\langle g\rangle^{W_n}\). The rest of the proof is obtained by iterating inductively this argument.
%
%
\end{proof}
A straightforward consequence is the following estimate.
\begin{corollary}
	If \(N\trianglelefteq W_n\) is a normal subgroup such that \(N\subseteq (B_k\cdots B_n)\setminus (B_{k+1}\cdots B_n)\), then, \(N\) contains \(\gamma_{p^{k-1}+1}(W_n)\) and 
	\[|N:\gamma_{p^{k-1}+1}(W_n)|\le \bigl(p^{p^{k-1}}\bigl)^{n-k+1}.\]
	In particular this index is bounded above by a function depending only on \(p\) and \(k\).
\end{corollary}
\begin{remark}
	Notice that if \(k=n\), then \(N\) coincides with a term of the lower central series depending only on the maximal monomial term appearing in \(N\).
\end{remark}

\section{The associated Lie algebra of \(W_n\)}\label{sec:lie}
In this section, we introduce the Lie algebra \(\Lie_n\) associated to the group \(W_n\) and we define a map between this structures. We also compute the upper and the lower central series of \(\Lie_n\).

We start noting that each base subgroup \(B_i\) is a uniserial module for \(W_n\) so that by Corollary~\ref{cor:desc_satur} the quotient of two consecutive terms of the lower central series is an elementary abelian group.
This implies that the graded Lie ring  $\Lie_n$ associated to the lower central series of $W_n$ is a Lie algebra over the field \(\F_p\). 
In \cite{netreba} this Lie algebra is characterized as an iterated wreath product $\Lie_n=\wr^n\Lie_1$, where $\Lie_1$ is the one dimensional algebra over $\F_p$. In particular, let $\partial_k$ be the derivation given by the standard partial derivative with respect to the variable $x_k$, with $1\le k\le n$. We identify $\Lie_n$ as the subalgebra of the Witt algebra over \(\F_p\) in \(n\) variables (see \cite[Chapter~2]{strade}) spanned by the basis
\begin{equation*}
	\mathfrak{B}=\bigcup_{k=1}^n \mathfrak{B}_k
	\text{ where }
	\mathfrak{B}_k=\Set{x^\Lambda\partial_k\mid \Lambda\in \mathcal{P}_p(k-1)}.
\end{equation*}
The product of $\Lie_{n}$ is defined on the basis \(\mathfrak{B}\) as follows
\begin{eqnarray*}
	\left[x^\Lambda \partial_k , x^\Theta \partial_j\right] :=&
	\partial_{j}(x^\Lambda)  x^\Theta \partial_k - x^\Lambda \partial_{k}(x^\Theta) \partial_j \nonumber \\
	=&
	\begin{cases}
		\partial_{j}(x^\Lambda)  x^\Theta \partial_k & \text{if \(j<k\)}, \\
		- x^\Lambda \partial_{k}(x^\Theta) \partial_j & \text{if \(j>k\)},\\
		0 & \text{otherwise.}
	\end{cases}
\end{eqnarray*} 
This operation is then extended by bilinearity on the whole \(\Lie_n\).

\medskip
Let $cx^\Lambda\Delta_k\in W_n$, where \(c\in \F_p\), we define 
\begin{align*}
	\phi_i(c x^\Lambda\Delta_k)=
	\begin{cases}
		cx^\Lambda\partial_k&\text{ if }x^\Lambda\Delta_k\in \gamma_i(W_n)\setminus \gamma_{i+1}(W_n)\\
		0&\text{otherwise}
	\end{cases}
\end{align*}
Notice that $\phi_i(x^\Lambda\Delta_k)\neq 0$ if and only if $\pdeg(x^\Lambda\Delta_k)=p^{n-1}-i$. Let $f\Delta_k$ be an homogeneous element of $\gamma_i(W_n)$, we define $\phi_i(f\Delta_k)\defeq \phi_i(\lt(f\Delta_k))$. Let $g=g_1\dots g_n\in W_n$, we set
\begin{equation*}
	\phi_i(g)=\begin{cases}
	\sum_{j=1}^n \phi_i(\lt(g_j)) & \text{if } g\in \gamma_{i}(W_n)\setminus\gamma_{i+1}(W_n),\\
	0 & \text{otherwise.}
	\end{cases}
\end{equation*} 
More generally we define \(\phi \colon W_n\to \Lie_n\) by setting \(\phi(g)=\phi_i(g)\) if \(g\in \gamma_i(W_n)\setminus \gamma_{i+1}(W_n)\).
	\begin{definition}\label{def:homogeneous}
	A Lie subring $\mathfrak{h}$ of $\Lie_n$ is said to be homogeneous if it is the span over $\F_p$ of some subset $\mathfrak{H}$ of $\mathfrak{B}$.
\end{definition}
Let $S$ be a saturated subgroup of $W_n$. We shall denote by $S^\phi$ the homogeneous Lie subring of $\Lie_n$ spanned by the set $\phi(S\cap \mathcal{B})$.


\begin{lemma}\label{lem:phi i+j}
	Let  $x^\Lambda\Delta_k\in \gamma_i(W_n)\setminus \gamma_{i+1}(W_n)$ and $x^\Theta\Delta_h\in \gamma_j(W_n) \setminus \gamma_{j+1}(W_n)$.
	If the commutator \([x^\Lambda\Delta_k,x^\Theta\Delta_h]\) is not trivial, then it lies in \(\gamma_{i+j}(W_n)\setminus \gamma_{i+j+1}(W_n)\) and
	the following equality holds 
	\begin{equation}
	   \phi_{i+j}([x^\Lambda\Delta_k,x^\Theta\Delta_h])=[\phi_i(x^\Lambda\Delta_k),\phi_j(x^\Theta\Delta_h)].
	\end{equation}
\end{lemma}
\begin{proof}
	Without lost of generality we can assume $k>h$, that \(\phi_i(x^\Lambda\Delta_k)\ne 0\ne \phi_j(x^\Theta\Delta_h)\) and that \(\dfrac{\partial x^\Lambda}{\partial x_h}\ne 0\). Notice that  $\pdeg([x^\Lambda\Delta_k,x^\Theta\Delta_h])=\pdeg\bigl(\dfrac{\partial x^\Lambda}{\partial x_h}x^\Theta\Delta_k\bigr)=p^{n-1}-i-j$, thus
	\begin{align*}
		\phi_{i+j}([x^\Lambda\Delta_k,x^\Theta\Delta_h])&=\phi_{i+j}\left(\dfrac{\partial x^\Lambda}{\partial x_h}x^\Theta\Delta_k\right)\\
		&=\dfrac{\partial x^\Lambda}{\partial x_h}x^\Theta\partial_k\\
		&=[x^\Lambda\partial_k,x^\Theta\partial_h]\\
       &=[\phi_i(x^\Lambda\Delta_k),\phi_j(x^\Theta\Delta_h)]. \qedhere
	\end{align*}
\end{proof}

\begin{remark}\label{rmk: |S|}
	A straightforward consequence of the previous lemma is that if $S$ is a saturated subgroup of $W_n$, then 
	\(|S|=|S^\phi|\).
\end{remark}

\begin{corollary}
	For each $i\ge 1$ we have $\gamma_i(W_n)^\phi=\Lie_n^i$ is the \(i\)-th Lie power of \(\Lie_n\).
\end{corollary}
We now compute the upper central series of \(\Lie_n\). 
\begin{definition}
Let  $1\le j\le n$. We define  $\xi_m$ as the $\F_p$-span of the set $$\Set{x^{\Lambda_n}\partial_n,\dots,x^{\Lambda_1}\partial_1 \mid \pdeg(x^{\Lambda_i}\partial_i)< m\text{ for all }i=1,\dots,n}.$$
	\end{definition}
 Notice that, by Lemma~\ref{lem:gammasuiB}, \(\xi_m=\gamma_{p^{n-1}-m}(W_n)^\phi\),
and so
		\(\xi_m\subseteq \xi_{m+1}\).
	\begin{lemma}
		For each $x^\Theta\partial_\ell\in \Lie_n$ and $x^\Lambda\partial_i\in \xi_m$ the commutator $[x^\Lambda\partial_i, x^\Theta\partial_\ell]$ belongs to $\xi_{m-1}$.
	\end{lemma}
	\begin{proof}
		Without loss of generalities, we may assume $x^\Lambda\partial_i\neq 0 \neq x^\Theta\partial_\ell$.
		For $\ell<i$, the assertion is easily verified noting that 
		$$\pdeg\bigl(\dfrac{\partial x^\Lambda}{\partial x_\ell}x^\Theta\partial_i\bigr)\le\pdeg(x^\Lambda\partial_i)-1<m-1.$$ 
		If $\ell>i$, we have that $[x^\Theta\partial_\ell,x^\Lambda\partial_i]=\dfrac{\partial x^\Theta}{\partial x_i}x^\Lambda\partial_\ell$. We observe that, since $x^\Lambda\partial_i\neq 0$, we must have $m> \mu_i$. Hence 
		$$\pdeg\bigl(\dfrac{\partial x^\Theta}{\partial x_i}x^\Lambda\partial_\ell\bigr)\le p^{n-1}-1-p^{i-1}=\mu_i-1< m-1. $$
		Thus, \(x^\Lambda\partial_i\in \xi_{m-1} \).
	\end{proof}
	\begin{corollary}
		For all $i=1,\dots,n$ we have $\xi_m\cap \mathfrak{B}_i\subseteq Z_m(\Lie_n)\cap \mathfrak{B}_i$.
	\end{corollary}
		
		\begin{lemma}
		Let $x^\Lambda\partial_i$ be such that $\pdeg(x^\Lambda\partial_i)=r+\mu_i$. Then for each  $k+1<r$ there exists $x^\Theta\partial_\ell\in \Lie_n$ such that $\pdeg([x^\Theta\partial_\ell,x^\Lambda\partial_i])>k+\mu_i$.
	\end{lemma}
	\begin{proof}
		Let $r=\lambda_1+\lambda_2p+\dots +\lambda_{i-1}p^{i-2}$ and $k+1=\gamma_1+\gamma_2p+\dots+\gamma_{i-1}p^{i-2}$. Let $j$ be the maximum index such that $\lambda_j>\gamma_j$. If $\lambda_j-\gamma_j>1$, then
		\begin{equation*}
			\pdeg([x^\Lambda\partial_i,\partial_j])>k+\mu_i.
		\end{equation*}
		If $\lambda_j-\gamma_j=1$ and there exists $s<j$ such that $\lambda_s\neq0$, then $\pdeg([x^\Lambda\partial_i,\partial_s])>k+\mu_i$. If such $s$ does not exist and $j\neq 1$, then $\pdeg([x^\Lambda\partial_i,x_1^{p-1}\dots x_{j-1}^{p-1}\partial_j])=r-1+\mu_i>k+\mu_i$. If $j=1$, then $\pdeg([x^\Lambda\partial_i,\partial_1])=r-1+\mu_i>k+\mu_i$.
	\end{proof}
	\begin{corollary}
		For all $i=1,\dots,n$ the following equalities hold $$\xi_m\cap \mathfrak{B}_i=Z_m(\Lie_n)\cap \mathfrak{B}_i=\Lie_n^{p^{n-1}-m}\cap \mathfrak{B}_i.$$ In particular, 
		\(\xi_m=Z_m(\Lie_n)=\Lie_n^{p^{n-1}-m}\).
	\end{corollary}
	\medskip

%
%
%
%
%

%
%
%

	\section{A chain of normalizers}\label{sec:chain}
	In this section, we compute the growth of the normalizer chain $\lbrace N_i\rbrace_{i\geq -1}$ starting from the canonical elementary abelian regular subgroup \(T=\langle \Delta_1,\dots,\Delta_n\rangle\le W_n\),  and defined as follows
	\begin{align}\label{eq:normalizerchain}
		N_i=\begin{cases}
			T&\text{ if }i=-1\\
			N_{W_n}(N_{i-1})&\text{ if }i\ge 0.
		\end{cases}
	\end{align}  
	\begin{proposition}\label{prop:normsaturo}
		Let $S\le W_n$ be a saturated subgroup. The normalizer $N_{W_n}(S)$ of $S$ in $W_n$ is a saturated subgroup.
	\end{proposition}
	\begin{proof}
		Let  $g=hg_k\in N_{W_n}(S)$ with $ g_k\in B_k\setminus \Set{1}$ and $h\in B_{k+1}\cdots B_n$. Since $S$ is saturated, we have that the condition \(g\in N_{W_n}(S)\) is equivalent to require that \([g,s]\in S\) for every $s\in B_i\cap S$ and all $i\in \Set{1,\dots,n}$. Notice that in order to prove that \(N_{W_n}(S)\) is saturated,  it is enough to prove that $g_k\in N_{W_n}(S)$ since then \(h=gg_k^{-1}\in N_{W_n}(S)\) and we can argue by induction on \(k\). Let \(s\in B_i\cap S\), we have 
		\begin{equation*}
			S\ni [g,s]=[h,s]^{g_k}[g_k,s].
		\end{equation*}
		
		If $k=i$, then $[g_k,s]=1$ and we are done.
		
		Without lost of generality we can suppose $[g_k,s]\neq 1$.
		If $k> i$, then $[h,s]^{g_k}\in B_{k+1}\cdots B_n$ and $[g_k,s]\in B_k$. In particular, since $S$ is saturated, we get $[g_k,s]\in S$.
		
		If $k<i$, let $\bar h \in B_{k+1}\cdots B_{i-1}$ and $\hat h\in B_{i}\cdots B_n$ such that $h=\hat h\bar h$. We have that
		\begin{equation*}
			S\ni [g,s]=[\hat h\bar h g_k,s]=[\hat h, s]^{\bar h g_k}[\bar h g_k,s]
		\end{equation*}
		and that $[\bar h g_k,s]\in B_i$, $[\hat h, s]^{\bar h g_k}\in B_{i+1}\cdots B_n$. It follows that $[\bar h g_k,s]\in B_i\cap S$. Let now consider the decomposition $\bar h=h_{i-1}\cdots h_{k+1}$ with $h_j\in B_j$. We get that
		\begin{equation*}
			S\ni [\bar h g_k,s]=[h_{i-1},s]^{h_{i-2}\cdots h_{k+1}g_k}\cdots [h_{k+1},s]^{g_k}[g_k,s].
		\end{equation*}
		Notice that, on the one hand the variable $x_j$ appears in each monomial element of $[g_k,s]$ with degree equal to the degree of the variable $x_j$ in $s$. On the other hand, the variable $x_j$ appears in each monomial element of $[h_j,s]^{h_{j-1}\dots h_{k+1}g_k}$ with degree strictly less than the degree of the variable $x_j$ in $s$. Consequently, each monomial element in the decomposition of the commutator $[g_k,s]$ differs from each monomial element in the decomposition of the commutators $[h_j,s]^{h_{j-1}\dots h_{k+1}g_k}$, for $j={k+1},\dots i-1$.  Thus, since \(S\) is saturated, each monomial elements of the decomposition $[g_k,s]$ belongs to $S$, proving the statement.
	\end{proof}
	Bearing in mind  the study of the idealizer chain introduced in~\cite{modular}, we will compute the growth normalizer chain of Equation~\eqref{eq:normalizerchain} using the function \(\phi\) previously introduced. 
	\begin{definition}
	If $\mathfrak{U}$ is a subset of $\mathfrak{B}$, then its idealizer is deﬁned as
	$$\mathfrak{N}_\mathfrak{B} (\mathfrak{U}) = \Set{b \in \mathfrak{B} \mid [b, u] \in \F_p \mathfrak{U} \text{ for all } h \in \mathfrak{U}}.$$
	\end{definition}
	
	We refer to \cite[Theorem~2.5]{modular} for a proof of the following result.
	\begin{theorem}
	Let $\mathfrak{H}$ be a homogeneous subring of $\Lie_n$ having basis $\mathfrak{U}\subseteq \mathfrak{B}$. The
	idealizer $\mathfrak{N}_{\Lie_n}(\mathfrak{H})$ of $\mathfrak{H}$ in $\Lie_n$ is the homogeneous subring of $\Lie_n$ spanned over $\F_p$ by $\mathfrak{N}_\mathfrak{B} (\mathfrak{U})$.
	\end{theorem}
	The following result is a technical lemma aiming to intertwine idealizers and normalizers.
	\begin{lemma}\label{lem: ciserve}
		Let $H$ be a saturated subgroup of $W_n$ and let $n\in B_\ell$. If $\lt([n,h])\in H$ for all $h\in H$, then $[n,h]\in H$.
	\end{lemma}
	\begin{proof}
		Since $H$ is a saturated subgroups, without loss of generality we may assume that \(h=g\Delta_i\in B_i\) for some \(i\), and  $n=f\Delta_\ell$.
		If $\ell>i$, then
		\begin{equation*}
			[n,h]=\sum_{s=1}^{p-1} \dfrac{1}{s!}\dfrac{\partial^s f}{\partial x_i^s}g^s\Delta_\ell
		\end{equation*}
		and $\lt([n,h])=\dfrac{\partial f}{\partial x_i}g\Delta_\ell\in H$. The statement follows noting that $ [\lt([n,h]),h]\in H$ so that \(\dfrac{\partial^s f}{\partial x_i^s}g^s\Delta_\ell \in H \) for  \(s=2,\dots,p-1\).
		If $\ell<i$, then
		\begin{equation*}
			[n,h]=\sum_{s=1}^\infty \dfrac{1}{s!}\dfrac{\partial^s g}{\partial x_\ell^s}f^s\Delta_i
		\end{equation*}
		and $\lt([n,h])=\dfrac{\partial g}{\partial x_\ell}f\Delta_i\in H$. By hypothesis, $\lt\bigl([n,\lt([n,h])]\bigr)=\dfrac{\partial^2 g}{\partial x_\ell^2}f^2\Delta_i\in H$. Iterating the process we obtain the desired result.
	\end{proof}
\begin{proposition}\label{prop:normideal}
	Let $H$ be a saturated subgroup of $W_n$. The following equality holds
	\begin{equation*}
		(N_{W_n}(H))^\phi=\mathfrak{N}_{\Lie_n}(H^\phi).
	\end{equation*}
\end{proposition}
\begin{proof}
	By Proposition~\ref{prop:normsaturo} we know that $N_{W_n}(H)$ is a saturated subgroup of $W_n$. Let $n\in N_{W_n}(H)\cap \mathcal{B}$ and $i$ an integer such that $n\in \gamma_i(W_n)\setminus\gamma_{i+1}(W_n)$. For every $h\in H\cap \mathcal{B}$, there exists an integer $j$ such that $h\in \gamma_j(W_n)\setminus \gamma_{j+1}(W_n)$ and we have the following equality by Lemma~\ref{lem:phi i+j}
	\begin{equation*}
		\phi([n,h])=\phi_{i+j}([n,h])=[\phi_i(n),\phi_j(h)].
	\end{equation*}
Since $\phi([n,h])\in H^\phi$ for all $h\in H\cap\mathcal{B}$, it follows that $\phi_i(n)=\phi(n)\in \mathfrak{N}_{\Lie_n}(H^\phi)$.
	
    We prove now the opposite inclusion. Let \(t \in \mathfrak{N}_{\Lie_n}(H^\phi)\cap \mathfrak{B}\). For some positive integer \(i\), there exists \(n\in \mathcal{B}\cap \bigl(\gamma_i(W_n)\setminus \gamma_{i+1}(W_n)\bigr)\) such that \(\phi(n)=t\). For all $\phi(h)\in H^\phi\cap\mathfrak{B}$ there exists an integer $j$ such that $\phi(h)=\phi_j(h)$ and 
    \begin{equation*}
    	H^\phi\ni [\phi_i(n),\phi_j(h)]=\phi_{i+j}([n,h]=\phi(\lt([n,h]))).
    \end{equation*}
    Thus, $\lt([n,h])\in H$ for all $h\in H$ and, by Lemma~\ref{lem: ciserve}, we have $[n,h]\in H$.
\end{proof}
\begin{remark}
By Proposition~\ref{prop:normideal}, we obtain that the correspondence sending \(H\) to \(H^\phi\) maps normal saturated subgroups of \(W_n\) into homogeneous ideals of \(\Lie_n\). Moreover, we define a new map \(\epsilon\colon \mathfrak{B}\to \mathcal{B}\) by \(x^\Lambda\partial_k\mapsto x^\Lambda\Delta_k\). If \(\mathfrak{I}\) is an homogeneous ideal of \(\Lie_n\), we denote by \(\mathfrak{I}^\epsilon\) the saturated subgroup of \(W_n\) generated by \(\epsilon(\mathfrak{I}\cap \mathfrak{B})\). Since \[\phi[\epsilon(x^\Lambda\partial_k),\epsilon(x^\Theta\partial_h)]=[\phi\epsilon(x^\Lambda\partial_k),\phi \epsilon(x^\Theta\partial_h)]=[x^\Lambda\partial_k,x^\Theta\partial_h],\] it follows that \(\mathfrak{I}^\epsilon\) is a normal saturated subgroup of \(W_n\) such that \( \bigl(\mathfrak{I}^\epsilon\bigr)^\phi=\mathfrak{I}\). Similarly, if \(N\) is a saturated normal subgroup of \(W_n\), then \(\bigl(N^\phi\bigr)^\epsilon=N\). This shows that the maps \((\cdot)^\phi\) and \((\cdot)^\epsilon\) realize a bijection between the poset of normal saturated subgroups of \(W_n\) and the poset of homogeneous ideals of \(\Lie_n\).

\end{remark}

By Remark~\ref{rmk: |S|} and Proposition~\ref{prop:normideal} we have that
\begin{equation}\label{eq: idealizzanti e normalizzanti}
	|\mathfrak{N}_{\Lie_n}(H^\phi)|=|N_{W_n}(H)|.
\end{equation}
Thus, the growth of the normalizer chain defined in Equation~\eqref{eq:normalizerchain} is equal to the growth of the following idealizer chain 
\begin{align*}
\mathfrak{N}_j=\begin{cases}
		\mathfrak{T}&\text{ if }j=-1\\
		\mathfrak{N}_{\Lie_n}(\mathfrak{N}_{j-1})&\text{ if }j\ge 0
	\end{cases}
\end{align*}
where $\mathfrak{T}$ is the homogeneous subring of $\Lie_n$ spanned by the set $\Set{\partial_1,\dots,\partial_n}$, which has been already described in \cite{modular}.
More in details, let $t_{p,i}$ be the number of partitions of $i$ into at least two parts, where each part can be repeated at most $p-1$ times, and 
\(q_{p,i}=\sum_{j=1}^it_{p,j}\).
The growth of the idealizer chain is then given in the following result.
\begin{theorem}\cite[Theorem~2.15]{modular}
		Let \(1\le      i\le      n-1\).  The $\F_p$-vector space
		 \(\mathfrak{N}_{i}/\mathfrak{N}_{i-1}\)  has
		dimension \(q_{p,i+1}\).
\end{theorem}
By way of Equation~\eqref{eq: idealizzanti e normalizzanti}, this theorem can be immediately restated in the group case 

\begin{theorem}
	Let \(1\le      i\le      n-1\), then $|N_i/N_{i-1}|=p^{q_{p,i+1}}$.
 
\end{theorem}

The first numbers of the sequences $t_{p,i}$ and $q_{p,i}$ for $p=3,5$ are available in OEIS~\cite{OEIS}, respectively under the labels~\href{https://oeis.org/A000726}{A000726} and~\href{https://oeis.org/A317910}{A317910}.

%

\bibliographystyle{abbrv} \bibliography{citation}

\end{document}